\documentclass[11pt,a4paper,reqno]{amsart}
\usepackage[english]{babel}
\usepackage[applemac]{inputenc}
\usepackage[T1]{fontenc}
\usepackage{palatino}
\usepackage{verbatim}
\usepackage{amsmath}
\usepackage{amssymb}
\usepackage{amsthm}
\usepackage{amsfonts}
\usepackage{graphicx}
\usepackage{mathtools}

\DeclarePairedDelimiter\floor{\lfloor}{\rfloor}
\usepackage[colorlinks = true, citecolor = black]{hyperref}
\author{Tuomas Orponen and Laura Venieri}
\title{A note on expansion in prime fields}
\address{University of Helsinki, Department of Mathematics and Statistics}
\email{tuomas.orponen@helsinki.fi}
\email{laura.venieri@helsinki.fi}

\thanks{T.O. and L.V. are supported by the Academy of Finland via the project \emph{Quantitative rectifiability in Euclidean and non-Euclidean spaces}, grant nos. $309365$ and $314172$}

\newcommand{\R}{\mathbb{R}}

\newcommand{\N}{\mathbb{N}}

\newcommand{\Z}{\mathbb{Z}}

\newcommand{\calL}{\mathcal{L}}

\newcommand{\calI}{\mathcal{I}}

\newcommand{\Hd}{\dim_{\mathrm{H}}}

\numberwithin{equation}{section}

\theoremstyle{plain}
\newtheorem{thm}[equation]{Theorem}
\newtheorem{conjecture}[equation]{Conjecture}
\newtheorem{lemma}[equation]{Lemma}

\newtheorem{ex}[equation]{Example}

\newtheorem{proposition}[equation]{Proposition}
\newtheorem{question}{Question}

\theoremstyle{definition}

\theoremstyle{remark}

\begin{document}

\begin{abstract} Let $\beta,\epsilon \in (0,1]$, and $k \geq \exp(122 \max\{1/\beta,1/\epsilon\})$. We prove that if $A,B$ are subsets of a prime field $\Z_{p}$, and $|B| \geq p^{\beta}$, then there exists a sum of the form
\begin{displaymath} S = a_{1}B \pm \ldots \pm a_{k}B, \qquad a_{1},\ldots,a_{k} \in A, \end{displaymath}
with $|S| \geq 2^{-12}p^{-\epsilon}\min\{|A||B|,p\}$.

As a corollary, we obtain an elementary proof of the following sum-product estimate. For every $\alpha < 1$ and $\beta,\delta > 0$, there exists $\epsilon > 0$ such that the following holds. If $A,B,E \subset \Z_{p}$ satisfy $|A| \leq p^{\alpha}$, $|B| \geq p^{\beta}$, and $|B||E| \geq p^{\delta}|A|$, then there exists $t \in E$ such that
\begin{displaymath} |A + tB| \geq c p^{\epsilon}|A|, \end{displaymath}
for some absolute constant $c > 0$. A sharper estimate, based on the polynomial method, follows from recent work of Stevens and de Zeeuw.
\end{abstract}

\maketitle

\section{Introduction}

The work in this note was motivated by the following problem in fractal geometry:
\begin{conjecture}\label{mainConjecture} Let $\alpha,\beta,\delta \in (0,1)$. Assume that $A,B,E \subset [0,1]$ are compact sets with $\dim A \leq \alpha$, $\dim B \geq \beta$ and $\dim B + \dim E \geq \dim A + \delta$. Then, there exists $t \in E$ such that
\begin{displaymath} \dim (A + tB) \geq \dim A + \epsilon \end{displaymath}
for some $\epsilon > 0$ depending only on $\alpha,\beta,\delta$.
\end{conjecture}
The conjecture follows from Bourgain's work \cite{Bo}, if $0 < \dim A = \dim B < 1$. Problems such as Conjecture \ref{mainConjecture} often have natural, and easier, analogues in the setting of finite fields. The same is true here, and one encounters the following question:
\begin{question}\label{Q1} Let $p \in \N$ be prime. Let $A,B,E \subset \Z_{p}$ be sets such that $|A| \leq p^{\alpha}$, $|B| \geq p^{\beta}$ and $|B||E| \geq p^{\delta}|A|$ for some $\alpha < 1$ and $\beta,\delta > 0$. Does there exist $t \in E$ such that $|A + tB| \geq p^{\epsilon}|A|$ for some $\epsilon = \epsilon(\alpha,\beta,\delta) > 0$? \end{question}
The following simple example motivates the requirements $\dim E + \dim B \geq \dim A + \delta$ and $|B||E| \geq p^{\delta}|A|$. 

\begin{ex} Consider the sets
\begin{displaymath} A = \left\{\tfrac{1}{n^{1/2}},\tfrac{2}{n^{1/2}},\ldots,\tfrac{n^{1/2}}{n^{1/2}}\right\} \end{displaymath}
and 
\begin{displaymath} B = \left\{\tfrac{1}{n^{1/4}},\tfrac{2}{n^{1/4}},\ldots,\tfrac{n^{1/4}}{n^{1/4}}\right\} = E. \end{displaymath}
for any integer $n = m^{4} \in \N$. Then $|B||E| = |A|$ and $BE \subset A$, so 
\begin{displaymath} |A + tB| \leq |A + EB| \leq |A + A| = 2|A| - 1, \qquad t \in E. \end{displaymath}
Iterating the construction above, it is not difficult to produce compact sets $A,B,E \subset [0,1]$ with $\Hd A = \tfrac{1}{2}$ and $\Hd B = \tfrac{1}{4} = \Hd E$ such that $\Hd (A + tB) = \Hd A$ for all $t \in E$. In $\Z_{p}$, an even easier example is given by $A = \{1,\ldots,\floor{p^{1/2}}\}$ and $B = \{1,\ldots,\floor{p^{1/4}}\} = E$.
\end{ex}

It turns out that the answer to Question \ref{Q1} is positive, and a good estimate follows easily from the recent incidence bound of Stevens and de Zeeuw, \cite[Theorem 4]{SdZ}:
\begin{proposition}\label{SdZProp} Assume that $A,B,E \subset \Z_{p}$ are sets with $|B| \leq |A|$ and $|B||E| \leq p$. Then, there exists $t \in E$ such that
\begin{displaymath} |A + tB| \gtrsim \min\{|A|^{2/3}(|B||E|)^{1/3},|A||B|,p\}. \end{displaymath}
\end{proposition}
For $A,B,E \subset \R$, by comparison, the Szemer\'edi-Trotter theorem gives $|A + tB| \gtrsim \min\{(|A||B||E|)^{1/2},|A||B|\}$ for some $t \in E$. Proposition \ref{SdZProp} is also closely related to \cite[Theorem 3]{RNRS}. Proposition \ref{SdZProp} certainly answers Question \ref{Q1}. However, the proofs in \cite{RNRS,SdZ} are based on the polynomial method, more precisely on a point-plane incidence bound in $\Z_{p}^{3}$ by Rudnev \cite{Ru}. It is not clear how to apply similar ideas to the continuous problem, Conjecture \ref{mainConjecture}, so a more elementary approach to Question \ref{Q1} seemed desirable. Here is the main result of this note:

\begin{thm}\label{main} Let $A,B \subset \Z_{p}$ and $\beta,\epsilon \in (0,1]$. Assume that $|B| > p^{\beta} > 4$. Then, there exists an integer
\begin{displaymath} k \leq \exp(C\max\{1/\beta,1/\epsilon\}), \end{displaymath}
and elements $a_{1},\ldots,a_{k} \in A$ such that
\begin{equation}\label{mainIneq} |a_{1}B \pm \ldots \pm a_{k}B| \geq cp^{-\epsilon}\min\{|A||B|,p\} \end{equation}
for certain choices of signs $\pm \in \{-,+\}$. Here $c \geq 2^{-12}$ and $C \leq 122$ are absolute constants.
\end{thm}

A positive answer to Question \ref{Q1} follows easily from Theorem \ref{main}, applied to $B$ and $E$, and combined with the Pl\"unnecke-Ruzsa inequalities. The proof of Theorem \ref{main} is elementary and does not use polynomials; instead, it consists of a reduction to a sum-product estimate of Bourgain, \cite[Lemma 2]{Bo2}, stating briefly that 
\begin{equation}\label{BourgainIneq} |8AB - 8AB| \gtrsim \min\{|A||B|,p\}. \end{equation}
Unfortunately, while the proof of \eqref{BourgainIneq} is elementary as well, it does not easily generalise to a "continuous" setting. So, at the end of the day, we are not much closer to proving Conjecture \ref{mainConjecture}.

The paper is organised so that Theorem \ref{main} is proven in Section \ref{mainProof}. The application to Question \ref{Q1}, as well as the proof or Proposition \ref{SdZProp}, is discussed in Section \ref{Q1Section}. 

\section{Acknowledgements} We are grateful to Changhao Chen for discussions during the early stages of the project, and for useful comments on the manuscript. We also thank Frank de Zeeuw and Oliver Roche-Newton for useful comments, and for pointing out references.

\section{Proof of the main theorem}\label{mainProof}

Before starting the proof of Theorem \ref{main}, we record the following lemma. It is quite likely well-known, and at least we extracted the argument from a paper of Bourgain, see \cite[(7.20)]{Bo}. 
\begin{lemma}\label{mainLemma} Let $(\mathbb{G},+)$ be an Abelian group, and assume that $A,B \subset \mathbb{G}$ are sets with $|A + A| \leq C_{1}|A|$ and $|B + B| \leq C_{2}|B|$. Assume moreover that there exists $G \subset A \times B$ with $|G| \geq |A||B|/C_{3}$ such that 
\begin{displaymath} |\pi_{1}(G)| \leq C_{4}|A|, \end{displaymath} 
where $\pi_{1}(x,y) = x + y$. Then $|A + B| \leq C|A|$ with $C = C_{1}C_{2}C_{3}C_{4}$.
\end{lemma}

\begin{proof} We start by observing that
\begin{equation}\label{form14} \chi_{A + B}(t) \leq \frac{1}{|G|} \sum_{(x,y) \in (A + A) \times (B + B)} \chi_{\pi_{1}(-G + (x,y))}(t), \qquad t \in \mathbb{G}. \end{equation}
Indeed, if $t \in A + B = \pi_{1}(A \times B)$, then $t = \pi_{1}(a,b)$ for some $(a,b) \in A \times B$. We then note that $(a,b) \in -G + (x,y)$ -- and hence $t = \pi_{1}(a,b) \in \pi_{1}(-G + (x,y))$ -- for all 
\begin{displaymath} (x,y) \in G + (a,b) \subset (A \times B) + (A \times B) = (A + A) \times (B + B). \end{displaymath}
So, \eqref{form14} follows from $|G + (a,b)| = |G|$. Finally,
\begin{align*} |A + B| & = \sum_{t \in \mathbb{G}} \chi_{A + B}(t)\\
& \leq \frac{1}{|G|} \sum_{(x,y) \in (A + A) \times (B + B)} \sum_{t \in \mathbb{G}} \chi_{\pi_{1}(-G + (x,y))}(t)\\
& \leq \frac{|\pi_{1}(G)||A + A||B + B|}{|G|} \leq \frac{C_{1}C_{2}C_{3}C_{4}|A|^{2}|B|}{|A||B|} = C|A|, \end{align*}
as claimed. \end{proof}

Now, we are ready to prove Theorem \ref{main}. 

\begin{proof}[Proof of Theorem \ref{main}] We may assume that $A,B \neq \emptyset$. Fix $n \in \N$ such that
\begin{equation}\label{form20} \max\left\{\tfrac{4}{\beta},\tfrac{58}{\epsilon}\right\} \leq n \leq 60\max\left\{\tfrac{1}{\beta},\tfrac{1}{\epsilon}\right\}, \end{equation}
and write $\delta := 1/n$. In particular
\begin{equation}\label{form17} \frac{p^{-2\delta}}{2}|B| \geq \frac{p^{-2/n + \beta}}{2} > \frac{p^{\beta/2}}{2} > 1. \end{equation}
Write $A_{1} := A$, and inductively
\begin{displaymath} A_{j + 1} := A_{j} + A_{j}, \qquad j \geq 1. \end{displaymath}
We note that there exists $1 \leq j \leq n + 1$ such that
\begin{equation}\label{form3} |A_{j} + A_{j}| = |A_{j + 1}| \leq p^{\delta}|A_{j}|, \end{equation}
since otherwise
\begin{displaymath} p \geq |A_{n + 2}| \geq p^{\delta}|A_{n + 1}| \geq \ldots \geq p^{\delta \cdot (n + 1)}|A_{1}| > p, \end{displaymath}
a contradiction. We define $\bar{A} := A_{j}$ for some index $j \leq n + 2$ satisfying \eqref{form3}. 

Next, in a similar spirit, we define a sequence of sets $H_{k}$, as follows. Start by setting $H_{1} := a_{1}B$ for any $a_{1} \in \bar{A}$. Next, assume that $H_{l}$ has already been defined for some $l \geq 1$. Choose $2^{l - 1}$ elements $a^{l}_{1},\ldots,a^{l}_{2^{l - 1}} \in \bar{A}$, and $(2^{l - 1} - 1)$ signs $\pm \in \{+,-\}$ such that
\begin{displaymath} H_{l + 1} := H_{l} \pm a^{l}_{1}B \pm \ldots \pm a^{l}_{2^{l - 1}}B \end{displaymath}
has maximal cardinality (among all such choices of $a_{1}^{l},\ldots,a_{2^{l - 1}}^{l}$, and choices of signs). As before, there exists $1 \leq l \leq n + 1$ such that
\begin{displaymath} |H_{l + 1}| \leq p^{\delta}|H_{l}|. \end{displaymath}
Now, we set $H := H_{l}$ for such an index $1 \leq l \leq n + 1$. We note that
\begin{equation}\label{HProp1} |H + H| \leq |H_{l + 1}| \leq p^{\delta}|H_{l}| = p^{\delta}|H|, \end{equation}
by the maximality of $|H_{l + 1}|$, since $H = H_{l}$ can be written as a sum of $2^{l - 1}$ terms of the form $a_{j}B$, $a_{j} \in \bar{A}$. It is even clearer that
\begin{equation}\label{HProp2} |H \pm aB| \leq |H_{l + 1}| \leq p^{\delta}|H_{l}| = p^{\delta}|H|, \qquad a \in \bar{A}. \end{equation}

Evidently $H$ is a set of the kind appearing on the left hand side of \eqref{mainIneq}; more precisely $H$ is a sum of at most 
\begin{displaymath} 2^{2(n + 1)} \leq \exp(122\max\{1/\beta,1/\epsilon\}) \end{displaymath}
terms of the form $a_{j}B$ with $a_{j} \in A$. It remains to show that $H$ satisfies \eqref{mainIneq}.

We start the proof by showing that there exists an element $b_{0} \in B$, and subset $B' \subset B$ of cardinality $|B'| \geq p^{-2\delta}|B|/2$ such that
\begin{equation}\label{form4} |H + (b_{0} - b)\bar{A}| \leq 2p^{4\delta}|H|, \qquad b \in B'. \end{equation}
To prove \eqref{form4}, we consider the following set $P \subset \Z_{p}^{2}$,
\begin{displaymath} P := \{(a,r) \in \Z_{p}^{2} : a \in \bar{A} \text{ and } r \in aB + H\}, \end{displaymath}
and we note that
\begin{equation}\label{form1} |\bar{A}||H| \leq |P| \leq p^{\delta}|\bar{A}||H| \end{equation}
by \eqref{HProp2}. Consider also the following family of lines: $\calL := \{\ell_{h,b}\}_{(h,b) \in H \times B}$, where
\begin{displaymath} \ell_{h,b} = \{(x,y) \in \Z_{p}^{2} : y = xb + h\}. \end{displaymath}
We note that every line in $\calL$ contains exactly $|\bar{A}|$ points in $P$. Indeed:
\begin{displaymath} P \cap \ell_{h,b} = \{(a,ab + h) : a \in \bar{A}\}, \qquad (h,b) \in H \times B.  \end{displaymath} 
It follows that that if $b \in B$ is fixed, the (disjoint) lines $\{\ell_{h,b}\}_{h \in H}$ cover $|\bar{A}||H|$ points of $P$ in total. In other words, the sets
\begin{displaymath} P_{b} := \bigcup_{h \in H} (P \cap \ell_{h,b}) \subset P, \qquad b \in B, \end{displaymath}
satisfy 
\begin{equation}\label{form16} |P_{b}| = |\bar{A}||H| \geq p^{-\delta}|P|, \qquad b \in B, \end{equation}
recalling \eqref{form1}. Using Cauchy-Schwarz in a standard way, see for example \cite[Lemma 4.2]{Gr}, it follows from \eqref{form16} that there exists $b_{0} \in B$, and a subset $B' \subset B$ with $|B'| \geq p^{-2\delta}|B|/2$ such that
\begin{equation}\label{form6} |P_{b} \cap P_{b_{0}}| \geq \frac{p^{-2\delta}}{2}|P|, \qquad b \in B'. \end{equation}
We record here that 
\begin{equation}\label{form15} \pi_{-b}(P_{b} \cap P_{b_{0}}) \subset \pi_{-b}(P_{b}) \subset H, \qquad b \in B, \end{equation}
where $\pi_{c}(x,y) = xc + y$. Indeed, if $p = (a,ab + h) \in P_{b}$ for some $a \in \bar{A}$ and $h \in H$, then
\begin{displaymath} \pi_{-b}(p) = -ab + ab + h = h \in H. \end{displaymath}
Now, given such a point $b_{0} \in B$, we define the following bijective linear map:
\begin{displaymath} T(x,y) := (x,y - b_{0}x). \end{displaymath}
It is immediate that
\begin{equation}\label{form5} \pi_{b_{0} - b}(T(x,y)) = \pi_{-b}(x,y), \qquad (x,y) \in \Z_{p}^{2}. \end{equation}
Moreover, $T(P_{b_{0}}) \subset \bar{A} \times H$. Indeed, if $p = (a,ab_{0} + h) \in P_{b_{0}}$ for some $a \in \bar{A}$ and $h \in H$, then
\begin{displaymath} T(p) = (a,ab_{0} + h - ab_{0}) = (a,h) \in \bar{A} \times H, \end{displaymath}
as claimed. We write
\begin{displaymath} G_{b} := T(P_{b} \cap P_{b_{0}}) \subset \bar{A} \times H, \qquad b \in B', \end{displaymath}
and conclude from \eqref{form6} that
\begin{equation}\label{form7} |G_{b}| \geq \frac{p^{-2\delta}}{2}|P| \geq \frac{p^{-2\delta}}{2}|\bar{A}||H|. \end{equation}
From \eqref{form5} and \eqref{form15}, we conclude that
\begin{displaymath} \pi_{b_{0} - b}(G_{b}) = \pi_{-b}(P_{b} \cap P_{b_{0}}) \subset H, \qquad b \in B', \end{displaymath}
so in particular $|\pi_{b_{0} - b}(G_{b})| \leq |H|$. Now, Lemma \ref{mainLemma} applied to the sets $H$, $(b_{0} - b)\bar{A}$ and $\{(y,(b_{0} - b)x) : (x,y) \in G_{b}\} \subset H \times (b_{0} - b)\bar{A}$ implies, recalling \eqref{form3}, \eqref{HProp1} and \eqref{form7}, that
\begin{displaymath} |H + (b_{0} - b)\bar{A}| \leq 2p^{4\delta}|H|, \qquad b \in B', \end{displaymath}
as claimed in \eqref{form4}.

Set $\bar{B} := b_{0} - B'$. Then $|\bar{B}| \geq p^{-2\delta}|B|/2 > 1$ by \eqref{form17}, and 
\begin{equation}\label{form8} |H \pm b\bar{A}| \leq 2p^{5\delta}|H| \quad \text{and} \quad |H \pm a\bar{B}| \leq p^{\delta}|H|, \qquad a \in \bar{A}, \: b \in \bar{B}, \end{equation}
combining \eqref{HProp2} and \eqref{form4}. The "$-$" inequality $|H - b\bar{A}| \leq 2p^{5\delta}|H|$ moreover uses \eqref{form3} and Ruzsa's triangle inequality:
\begin{displaymath} |H - b\bar{A}| \leq \frac{|H + b\bar{A}||b\bar{A} + b\bar{A}|}{|\bar{A}|} \leq 2p^{5\delta}|H|, \qquad b \in \bar{B}.  \end{displaymath}

  Now, we apply a result of Bourgain, namely \cite[Lemma 2]{Bo2}. It states that if 
\begin{equation}\label{form18} (\bar{A} - \bar{A}) \cap (\bar{B} - \bar{B}) \neq \{0\}, \end{equation}
then there exist subsets $\bar{B}_{1} \subset \bar{B}$, $Z \subset (\bar{A} - \bar{A}) \cap (\bar{B} - \bar{B})$, and elements $a_{1},\ldots,a_{6} \in \bar{A}$, $b_{1},\ldots,b_{6} \in \bar{B}$ such that  
\begin{equation}\label{form19} |(b_{1} - b_{2})A + (a_{1} - a_{2} + a_{3} - a_{4})\bar{B}_{1} + (a_{5} - a_{6} + b_{3} - b_{4} + b_{5} - b_{6})Z| \geq \frac{1}{2} \min\{|\bar{A}||\bar{B}|,p - 1\}. \end{equation}
The condition \eqref{form18} is not automatically satisfied, but in any case we can proceed as follows. Since Theorem \ref{main} is trivial for $|A| \leq 1$, we may assume that $|\bar{A}| \geq |A| \geq 2$. Since also $|\bar{B}| \geq 2$ by the choice of $\delta$ in \eqref{form17}, there exist $a,a' \in \bar{A}$ and $b,b' \in \bar{B}$ with $a \neq a'$ and $b \neq b'$. Then, writing $\xi := (a - a')/(b - b') \neq 0$, we have
\begin{displaymath} (\bar{A} - \bar{A}) \cap (\xi \bar{B} - \xi \bar{B}) \supset \{a - a'\}, \end{displaymath}
and so Bourgain's result is applicable to $\bar{A}$ and $\xi\bar{B}$. Then, \eqref{form19} implies the existence of $a_{1},\ldots,a_{6} \in \bar{A}$ and $b_{1},\ldots,b_{6} \in \bar{B}$ such that the sum
\begin{displaymath} \xi(b_{1} - b_{2})\bar{A} + (a_{1} - a_{2} + a_{3} - a_{4})\xi \bar{B} + (a_{5} - a_{6})(\xi \bar{B} - \xi \bar{B}) + \xi(b_{3} - b_{4} + b_{5} - b_{6})(\bar{A} - \bar{A}) \end{displaymath}
has cardinality at least $\min\{|A||B|,p - 1\}/2$. Then, the same conclusion follows automatically for the sum 
\begin{align*} (b_{1} - b_{2})\bar{A} & + (a_{1} - a_{2} + a_{3} - a_{4}) \bar{B} + (a_{5} - a_{6})(\bar{B} - \bar{B}) + (b_{3} - b_{4} + b_{5} - b_{6})(\bar{A} - \bar{A})\\
& \subset b_{1}\bar{A} - b_{2}\bar{A} + b_{3}\bar{A} - b_{4}\bar{A} + b_{5}\bar{A} - b_{6}\bar{A} - b_{3}\bar{A} + b_{4}\bar{A} - b_{5}\bar{A} + b_{6}\bar{A}\\
&\quad + a_{1}\bar{B} - a_{2}\bar{B} + a_{3}\bar{B} - a_{4}\bar{B} + a_{5}\bar{B} - a_{6}\bar{B} - a_{5}\bar{B} + a_{6}\bar{B} =: \Sigma_{a_{1}\cdots a_{6}}^{b_{1}\cdots b_{6}}(\bar{A},\bar{B}).  \end{align*}
(In fact, Bourgain uses the same argument to prove the second part of \cite[Lemma 2]{Bo2}.) Finally, using \eqref{form8} and the Pl\"unnecke-Ruzsa inequalities for different summands (see \cite{Ruz} or \cite[Theorem 6.1]{Ruz2}), we conclude that 
\begin{displaymath} \frac{1}{2}\min\{|A||B|,p - 1\} \leq |\Sigma_{a_{1}\cdots a_{6}}^{b_{1}\cdots b_{6}}(\bar{A},\bar{B})| \leq 2^{10}p^{58\delta}|H| \leq 2^{10}p^{\epsilon}|H|. \end{displaymath}
The final inequality follows from \eqref{form20}. Hence,
\begin{displaymath} |H| \geq 2^{-12} p^{-58\delta} \min\{|A||B|,p\} \geq 2^{-12}p^{-\epsilon}\min\{|A||B|,p\} \end{displaymath}
as required. The proof of Theorem \ref{main} is complete.
\end{proof} 

\section{Lower bounds for $|A + tB|$}\label{Q1Section}

We start by applying Theorem \ref{main} to Question \ref{Q1}. Then, we recall the result of Stevens and de Zeeuw from \cite{SdZ} and apply it to prove Proposition \ref{SdZProp}.

Let $A,B,E \subset \Z_{p}$ be as in Question \ref{Q1}, with $|A| \leq p^{\alpha}$, $|B| \geq p^{\beta}$ and $|B||E| \geq p^{\delta}|A|$, and assume that $|A + tB| \leq p^{\epsilon}|A|$ for all $t \in E$, and for some $\epsilon > 0$. Then, for any $t_{1},\ldots,t_{k} \in E$, it follows from the Pl\"unnecke-Ruzsa inequalities for different summands (see \cite{Ruz} or \cite[Theorem 6.1]{Ruz2})  that there exists a non-empty subset $X \subset A$ with the property that
\begin{displaymath} |X + t_{1}B + \ldots + t_{k}B| \lesssim_{k} p^{k\epsilon}|X|. \end{displaymath}
By another application of the Pl\"unnecke-Ruzsa inequalities,
\begin{equation}\label{form21} |(t_{1}B + \ldots + t_{k}B) - (t_{1}B + \ldots + t_{k}B)| \lesssim_{k} p^{2k\epsilon}|A|. \end{equation}
Now, write $\epsilon' := \min\{\delta/2, (1 - \alpha)/2\}$, and use Theorem \ref{main} to choose $t_{1},\ldots,t_{k} \in E$, with $k \leq 2^{C\max\{1/\beta,1/\epsilon'\}}$, so that
\begin{displaymath} |t_{1}B \pm \ldots \pm t_{k}B| \gtrsim p^{-\epsilon'}\min\{|B||E|,p\} \geq \min\{p^{\delta/2},p^{(1 - \alpha)/2}\}|A| \end{displaymath}
for certain signs $\pm \in \{+,-\}$. It follows from \eqref{form21} that $p^{2k\epsilon}|A| \gtrsim \min\{p^{\delta/2},p^{(1 - \alpha)/2}\}|A|$, which gives a lower bound for $\epsilon$, depending only on $\alpha,\beta,\delta$.

To prove Proposition \ref{SdZProp}, we recall the statement of \cite[Theorem 4]{SdZ}:
\begin{thm}[Stevens-de Zeeuw]\label{SdZThm} Let $A,B \subset \Z_{p}$ be sets, and let $\calL$ be a collection of lines in $\Z_{p}^{2}$, with 
\begin{displaymath} |B| \leq |A|, \quad |A|^{2}|B| \leq |\calL|^{3}, \quad \text{and} \quad |B||\calL| \leq p^{2}. \end{displaymath} 
Then, the the set of incidences $\calI(A \times B,\calL) := \{(a,b,\ell) : (a,b) \in (A \times B) \cap \ell \text{ and } \ell \in \calL\}$ has cardinality at most
\begin{displaymath} |\calI(A \times B,\calL)| \lesssim |A|^{1/2}|B|^{3/4}|\calL|^{3/4} + |\calL|. \end{displaymath}
\end{thm}

\begin{proof}[Proof of Proposition \ref{SdZProp}] Recall that $|B| \leq |A|$ and $|B||E| \leq p$. Assume that
\begin{displaymath} |A + tB| \leq N, \qquad t \in E, \end{displaymath}
The aim is to prove that $N \gtrsim \min\{|A|^{2/3}(|B||E|)^{1/3},|A||B|,p\}$. Consider the family of lines
\begin{displaymath} \calL_{t} := \{\{(x,y) \in \Z_{p}^{2} : x = r - ty\}\}_{r \in A + tB}. \end{displaymath}
Write also $\calL := \cup \{\calL_{t} : t \in E\}$, so that $|A||E| \leq |\calL| \leq |E|N$. Assuming that $N \leq p$, as we may, we see that the hypotheses of Theorem \ref{SdZThm} are valid:
\begin{displaymath} |A|^{2}|B| \leq |A|^{3} \leq |\calL|^{3} \quad \text{and} \quad |B||\calL| \leq |B||E|N \leq p^{2}, \end{displaymath} 
Note that every point $(a,b) \in A \times B$ is incident to $|E|$ lines in $\calL$, since
\begin{displaymath} (a,b) \in \{(x,y) \in \Z_{p}^{2} : x = (a + tb) - ty\} \in \calL_{t}, \qquad t \in E. \end{displaymath}
It follows from this, and Theorem \ref{SdZThm}, that
\begin{displaymath} |A||B||E| \leq |\calI(A \times B,\calL)| \lesssim |A|^{1/2}|B|^{3/4}(|E|N)^{3/4} + |E|N. \end{displaymath}
If the second term is larger, we have $N \gtrsim |A||B|$, and the proof is complete. If the first term is larger, re-arranging the inequality gives
\begin{displaymath} N \gtrsim |A|^{2/3}(|B||E|)^{1/3}, \end{displaymath}
as desired. The proof of Proposition \ref{SdZProp} is complete. \end{proof}

\end{document}